\documentclass{amsart}
\usepackage{amssymb}
\usepackage[francais,english]{babel}
\title[Binomial Character Sums Modulo Prime Powers]{Binomial
 Character Sums Modulo Prime Powers}

\author{Vincent Pigno}
\address{ Department of Mathematics\\
          Kansas State University\\
          Manhattan, KS 66506}
\email{pignov@math.ksu.edu}
\author{Christopher Pinner}
\address{ Department of Mathematics\\
         Kansas State University\\ and
         Manhattan, KS 66506}
\email{pinner@math.ksu.edu}
\date{\today}

\newcommand{\be}{\begin{equation}}
\newcommand{\ee}{\end{equation}}
\newcommand{\ba}{\begin{align}}
\newcommand{\ea}{\end{align}}
\newcommand{\sumstar}[2]{\sideset{}{^*}\sum_{#1}^{#2}}
\begin{document}
\newenvironment{poliabstract}[1]
  {\renewcommand{\abstractname}{#1}\begin{abstract}}
  {\end{abstract}}
%\selectlanguage{francais}

% \renewcommand{\abstractname}{R�sum�}

% \begin{abstract}
% On montre que les sommes binomiales et li\'{e}es de caract\`{e}res multiplicatifs
% $$ \sum_{\stackrel{x=1}{(x,p)=1}}^{p^m} \chi (x^l(Ax^k +B)^w),\hspace{3ex}
% \sum_{x=1}^{p^m} \chi_1 (x)\chi_2(Ax^k +B), $$
% ont une \'{e}valuation simple pour m suffisamment grand (pour $m\geq2$ si $p\nmid ABk$).
% \end{abstract}
% 
%
%\selectlanguage{francais}
%\begin{poliabstract}{resume} 
%blah blah blah...
%\end{poliabstract}
%\selectlanguage{english}
%\begin{poliabstract}{abstract} 
% here is the english
%\end{poliabstract}

%\selectlanguage{english}

% \renewcommand{\abstractname}{Abstract}
% \begin{abstract} We show that the  binomial and related 
% multiplicative character sums
% $$ \sum_{\stackrel{x=1}{(x,p)=1}}^{p^m} \chi (x^l(Ax^k +B)^w),\hspace{3ex}
% \sum_{x=1}^{p^m} \chi_1 (x)\chi_2(Ax^k +B), $$
% have a simple evaluation for large enough $m$ (for $m\geq 2$ if
% $p\nmid ABk$).

% \end{abstract}
%R\'{e}sum\'{e}:
\maketitle
\selectlanguage{french}
\begin{abstract}
On montre que les sommes binomiales et li\'{e}es de caract\`{e}res multiplicatifs
$$ \sum_{\stackrel{x=1}{(x,p)=1}}^{p^m} \chi (x^l(Ax^k +B)^w),\hspace{3ex}
\sum_{x=1}^{p^m} \chi_1 (x)\chi_2(Ax^k +B), $$
ont une \'{e}valuation simple pour $m$ suffisamment grand (pour $m\geq2$ si $p\nmid ABk$).
\end{abstract}

\selectlanguage{english}
\begin{abstract}
 We show that the  binomial and related 
 multiplicative character sums
 $$ \sum_{\stackrel{x=1}{(x,p)=1}}^{p^m} \chi (x^l(Ax^k +B)^w),\hspace{3ex}
 \sum_{x=1}^{p^m} \chi_1 (x)\chi_2(Ax^k +B), $$
 have a simple evaluation for large enough $m$ (for $m\geq 2$ if
 $p\nmid ABk$).
\end{abstract}

\newtheorem{theorem}{Theorem}[section]
\newtheorem{corollary}{Corollary}[section]
\newtheorem{lemma}{Lemma}[section]
\newtheorem{conjecture}{Conjecture}[section]

\section{Introduction}
For a multiplicative character $\chi$  mod $p^m$, and rational
functions $f(x), g(x)\in \mathbb{Z}(x)$  one can define the mixed
exponential sum, \be \label{generaldef}
S(\chi,g(x),f(x),p^m):=\sumstar{x=1}{p^m} \chi(g(x))e_{p^m}(f(x))
\ee where $e_{y}(x)=e^{2\pi i x/y}$ and $*$ indicates that we omit
any $x$ producing a non-invertible denominator in $f$ or $g$. When
$m=1$ such sums have Weil \cite{weil} type bounds; for example if
$f$ is a polynomial and the sum is non-degenerate then \be
\label{weil}\left| S(\chi,g(x),f(x),p) \right| \leq (\deg (f) +\ell
-1) \, p^{1/2}, \ee where $\ell$ denotes the number of zeros  and
poles of $g$ (see Castro \& Moreno \cite{castro} or  Cochrane \&
Pinner \cite{cpstep} for a treatment of the  general case).

When  $m\geq 2$ methods of Cochrane and Zheng \cite{esopp} (see also
\cite{cz} \& \cite{cz2}) can be used to reduce and simplify the
sums. For example we showed in \cite{PP} that the sums \be
\label{oldsums} \sum_{x=1}^{p^m} \chi(x)e_{p^m}(nx^k) \ee can be
evaluated explicitly when $m$ is sufficently large (for $m\geq 2$ if
$p\nmid nk$). We show here that the multiplicative character sums,
\be \label{newsums}
S^*(\chi,x^l(Ax^k+B)^w,p^m)=\sum_{\stackrel{x=1}{p\nmid x}}^{p^m}
\chi(x^l(Ax^k+B)^w) \ee similarly have a simple evaluation for large
enough $m$ (for $m\geq 2$ if  $p\nmid ABk$). Equivalently, for
characters $\chi_1$ and $\chi_2$ mod $p^m$ we define \be
\label{defSS} S(\chi_1,\chi_2,Ax^k+B,p^m)= \sum_{x=1}^{p^m}\chi_1
(x)\chi_2 (Ax^k+B). \ee These include the mod $p^m$ generalizations
of the classical Jacobi sums \be \label{Jacobi} J(\chi_1,\chi_2,p^m)
= \sum_{x=1}^{p^m} \chi_1(x)\chi_2(1-x). \ee These sums have been
evaluated exactly by Zhang Wenpeng \& Weili Yao  \cite{Wenpeng1}
when $\chi_1$, $\chi_2$ and $\chi_1\chi_2$ are primitive and $m\geq
2$ is even (some generalizations are considered in \cite{Wenpeng2}).

Writing \be \label{equiv} \chi_1=\chi^l,\;\; \chi_2=\chi^w,\;\;\;
\chi_1(x)\chi_2(Ax^k+B)=\chi(x^l(Ax^k+B)^w),\ee with $\chi_1=\chi_0$
the principal character if $l=0$, the correspondence between
\eqref{newsums} and \eqref{defSS} is clear. Of course the
restriction $p\nmid x$  in \eqref{newsums} only differs from
$\sum^*$ when $l= 0$. We shall assume throughout that $\chi_2$ is a
primitive character mod $p^m$ (equivalently $\chi$ is primitive and
$p\nmid w$); if $\chi_2$ is not primitive but $\chi_1$ is primitive
then $S(\chi_1,\chi_2,Ax^k+B,p^m)=0$ (since $\sum_{y=1}^p
\chi_1(x+yp^{m-1})=0$), if both are not primitive we can reduce to a
lower modulus
$S(\chi_1,\chi_2,Ax^k+B,p^m)=pS(\chi_1,\chi_2,Ax^k+B,p^{m-1}).$

It is interesting that the sums  \eqref{oldsums} and \eqref{newsums}  can both be written explicitly in terms
of classical Gauss sums for any $m\geq 1$. In particular one can
trivially recover the Weil bound in these cases. We explore this in
Section 2.

We assume, noting the correspondence \eqref{equiv} between
\eqref{newsums} and  \eqref{defSS}, that \be \label{defg}
g(x)=x^l(Ax^k+B)^w, \;\;\;p\nmid w \ee where $k,l$ are integers with
$k> 0$ (else $x\mapsto x^{-1}$) and
 $A$, $B$
non-zero integers  with \be \label{defn} A=p^nA', \;\; 0\leq n <m,
\;\;p\nmid A'B.\;\; \ee We define the integers $d\geq 1$ and $t\geq
0$  by \be \label{defdt} d=(k,p-1),\hspace{4ex} p^t \mid \mid k. \ee
For $m\geq n+t+1$ it transpires that the sum in \eqref{newsums} or
\eqref{defSS} is zero unless \be \label{conditions}
\chi_1=\chi_3^{k},\ee for some mod $p^m$ character, $\chi_3$ (i.e.\
$\chi$ is the $(k,\phi(p^m))/(k,l,\phi (p^m))$th power of a
character), and we have a solution, $x_0$, to a characteristic
equation of the form, \be \label{chareq0} g'(x) \equiv 0 \text{ mod
} p^{\text{min} \left\{ m-1,\,[\frac{m+n}{2}]+t\right\}} \ee with
\be \label{conditions2} p\nmid x_0(Ax_0^k+B).\ee
 Notice that in order to
have a solution to \eqref{chareq0} we must have \be
\label{conditions1} p^{n+t}\mid\mid l, \;\;\; p^{t}\mid\mid l+wk,
\ee if $m>t+n+1$ (equivalently $\chi_1$ is induced by a primitive
mod $p^{m-n-t}$ character and $\chi_1\chi_2^w$ is a primitive mod
$p^{m-t}$ character) and $p^{n+t}\mid l$ if $m=t+n+1$.

When \eqref{conditions} holds, \eqref{chareq0} has a solution $x_0$
satisfying \eqref{conditions2} and $m>n+t+1$, Theorem \ref{main}
below gives an explicit evaluation of the sum \eqref{defSS}. From this we
see that \be \label{abs} \left| \sum_{x=1}^{p^m}
\chi_1(x)\chi_2(Ax^k+B) \right| =
\begin{cases} dp^{m-1}, & \text{ if }
t+n+1 <m \leq 2t+n+2, \\
dp^{\frac{m+n}{2}+t}, & \text{ if } 2t+n+2 < m .\end{cases} \ee The
condition $m> t+n+1$ is natural here; if $t\geq m-n$ then one can of
course use Euler's Theorem to reduce the power of $p$ in $k$ to
$t=m-n-1$. If $t=m-n-1$ and the sum is non-zero then, as in a
Heilbronn sum, we obtain a mod $p$ sum, $p^{m-1}\sum_{x=1}^{p-1}
\chi (x^l(Ax^k+B)^w)$, where one does not expect a nice evaluation.
For $t=0$ the result \eqref{abs} can be obtained from \cite{esopp}
by showing equality in their $S_{\alpha}$ evaluated at the $d$
critical points $\alpha$. For $t>0$ the $\alpha$ will not have
multiplicity one as needed in \cite{esopp}.

Condition \eqref{conditions} will arise naturally in our proof of Theorem \ref{main} but can also be seen from elementary considerations.

\begin{lemma}\label{ChiPower}
For any odd prime $p$, multiplicative characters $\chi_1 $, $\chi_2$
mod $p^m$, and $f_1$, $f_2$ in $\mathbb Z [x]$, the sum
$\displaystyle S=\sum_{x=1}^{p^m} \chi_1(x)
\chi_2(f_1(x^k))e_{p^m}(f_2(x^k)) $ is zero unless $\chi_1
=\chi_3^k$ for some mod $p^m$ character $\chi_3$.
\end{lemma}

\begin{proof}
Taking $z = a^{\phi (p^m)/(k,\phi (p^m))}$, $a$ a primitive root mod
$p^m$, we have $z^{k}=1$ and
$$ S =\sum_{x=1}^{p^m} \chi_1(xz) \chi_2(f_1((xz)^k))e_{p^m}(f_2((xz)^k))= \chi_1 (z) S. $$ Hence if
$S\neq 0$ we must have $1=\chi_1(z)=\chi_1 (a)^{\phi (p^m)/(k,\phi
(p^m))}$ and $ \chi_1 (a) = e_{\phi(p^m)}\left(c'
(k,\phi(p^m))\right)$ for some integer $c'$. For an integer $c_1$
satisfying
$$c'(k,\phi(p^m))\equiv c_1 k \text{ mod }\phi(p^m), $$
we equivalently have $\chi_1=\chi_3^{k}$ where $\chi_3(a)=e_{\phi
(p^m)}(c_1)$.

\end{proof}
Finally we observe that if $\chi$ is a mod $rs$ character with
$(r,s)=1$, then $\chi=\chi_1\chi_2$ for a mod $r$ character $\chi_1$
and mod $s$ character $\chi_2$, and for any $g(x)$ in $\mathbb Z
[x]$
$$ \sum_{x=1}^{rs}\chi (g(x))=\sum_{x=1}^{r}\chi_1 (g(x))\sum_{x=1}^{s}\chi_2
(g(x)). $$ Thus it is enough to work modulo prime powers.

\section{Gauss Sums and Weil type bounds}
\noindent For a character $\chi$ mod $p^j$, $j\geq 1$, we let
$G(\chi,p^j)$ denote the classical Gauss sum
$$ G(\chi,p^j)=\sum_{x=1}^{p^j}\chi(x)e_{p^j}(x).  $$
Recall (see  for example Section 1.6 of Berndt, Evans \& Williams
 \cite{BerndtBk}) that \be \label{gaussabs}
\left|G(\chi,p^j)\right|=
\begin{cases} p^{j/2}, & \text{ if
$\chi$ is primitive mod $p^j$,} \\
 1, & \text{ if $\chi=\chi_0$ and $j=1$,} \\
  0, & \text{ otherwise. }
  \end{cases}\ee
It is well known that the mod $p$ Jacobi sums \eqref{Jacobi} (and
their generalization to finite fields) can be written in terms of
Gauss sums (see for example Theorem 2.1.3 of \cite{BerndtBk} or
Theorem 5.21 of \cite{LidlNiedBk}). This extends to the mod $p^m$
sums. For example when $\chi_1$, $\chi_2$ and $\chi_1\chi_2$ are
primitive mod $p^m$
\be \label{JacobiGauss} J(\chi_1,\chi_2,p^m)
=\frac{G(\chi_1,p^m)G(\chi_2,p^m)}{G(\chi_1\chi_2,p^m) },\ee and
$\left|J(\chi_1,\chi_2,p^m)\right|=p^{m/2}$ (see Lemma 1 of
\cite{Wenpeng3} or \cite{Wenpeng2}; the relationship for Jacobi sums
over more general residue rings modulo prime powers can be found in
\cite{JacobiSums}).

We showed in \cite{PP} that for $p\nmid n$ the sums
$$
S(\chi,x,nx^k,p^m)=\sum_{x=1}^{p^m} \chi(x) e_{p^m}(nx^k)
$$
are zero unless $\chi=\chi_1^k$ for some character $\chi_1$ mod
$p^m$, in which case (summing over the characters whose order
divides $(k,\phi(p^m))$ to pick out the $k$th powers)
$$ S(\chi,x,nx^k,p^m) = \sum_{\chi_2^{(k,\phi(p^m))}=\chi_0}
\overline{\chi_1 \chi_2}(n)G(\chi_1 \chi_2, p^m). $$
From this one immediately
obtains a Weil type bound $$ \left|S(\chi,x,nx^k,p^m)\right| \leq
(k,\phi (p^m)) p^{m/2}.
$$

From Lemma \ref{ChiPower} we know that the sum in  \eqref{defSS} is
zero unless $\chi_1=\chi_3^{k}$ for some character $\chi_3$ mod
$p^m$, in which case the sum can be written as  $(k,\phi (p^m))$ mod
$p^m$  Jacobi like sums $\sum_{x=1}^{p^m} \chi_5(x)\chi_2(Ax+B)$ and
again be expressed in terms of Gauss sums.

\begin{theorem}\label{GaussSum} Let $p$ be an odd prime. If $\chi_1$, $\chi_2$ are characters mod $p^m$
with $\chi_2$ primitive and  $\chi_1=\chi_3^{k}$ for some character
$\chi_3$ mod $p^m$, and $n$ and $A'$ are as defined in \eqref{defn},
then
$$ \sum_{x=1}^{p^m} \chi_1(x)\chi_2(Ax^k+B)= p^n \sum_{\chi_4\in X}
\overline{\chi_3\chi_4}(A')\chi_2 \chi_3 \chi_4 (B)\frac{
G(\chi_3\chi_4,p^{m-n})G(\overline{\chi_2 \chi_3
\chi_4},p^m)}{G(\overline{\chi_2},p^m)},$$ where $X$ denotes the mod
$p^m$ characters $\chi_4$ with $\chi_4^{D}=\chi_0$,  $D=(k,\phi
(p^m))$, such that $\chi_3\chi_4$ is a mod $p^{m-n}$ character.

\end{theorem}
We immediately obtain the Weil type bound \be \label{weil2}
\left|S(\chi_1,\chi_2, Ax^k+B,p^m)\right| \leq (k,\phi(p^m))
p^{(m+n)/2}. \ee For $m=1$ and $p\nmid A$ this gives us the bound
$$ \left| \sum_{x=1}^{p-1}\chi\left(x^l(Ax^k+B)^w\right) \right| \leq dp^{\frac{1}{2}}, $$
where $d=(k,p-1)$. For $l=0$ we can slightly improve this for the
complete sum,
$$\left| \sum_{x=0}^{p-1} \chi (Ax^k+B)\right| \leq (d-1)p^{\frac{1}{2}}, $$
since, taking $\chi_1=\chi_3=\chi_0$, $\chi_2=\chi$, the
$\chi_4=\chi_0$ term in Theorem \ref{GaussSum} equals $-\chi(B)$,
the missing $x=0$ term  in \eqref{newsums}. These correspond to the
classical Weil bound \eqref{weil} after an appropriate change of
variables to replace $k$ by $d$. For $m\geq t+1$ the bound
\eqref{weil2} is $dp^{\frac{m+n}{2}+t}$, so by \eqref{abs} we have
equality in \eqref{weil2} for $m\geq n+2t+2$, but not for $t+n+1
<m<2t+n+2$.

Notice that if $(k,\phi(p^m))=1$, as in the generalized Jacobi sums
\eqref{Jacobi}, with $\chi_2$ primitive, and $\chi_1=\chi_3^k$  is a mod $p^{m-n}$ character
if $p\mid A$, then we have the single $\chi_4=\chi_0$ term and
$$ \sum_{x=1}^{p^m} \chi_1 (x) \chi_2(Ax^k+B) = p^n
\overline{\chi}_{3}(A') \chi_2\chi_3(B)
\frac{G(\chi_3,p^{m-n})G(\overline{\chi_2\chi_3},p^{m})}{G(\overline{\chi_2},p^m)},
$$
of absolute value $p^{(m+n)/2}$ if $\chi_2,\chi_2\chi_3$ and $\chi_3$ are primitive mod $p^m$
and $p^{m-n}$ (noting that 
$\overline{G(\overline{\chi},p^m)}=\chi(-1)G(\chi,p^m)$ 
we plainly recover the form  \eqref{JacobiGauss} in that case).

For the multiplicative analogue of the classical Kloostermann sums,
$\chi$ assumed primitive and $p\nmid A$,
 Theorem
\ref{GaussSum} gives a sum of two terms of size $p^{m/2}$
$$ \sumstar{x=1}{p^m}\chi (Ax+x^{-1}) = \frac{\overline{\chi}_3(A)}{G(\overline{\chi},p^m)}
\left( G(\chi_3,p^m)^2 + \chi^*(A)
G\left(\chi_3\chi^*,p^m\right)^2\right)
$$
when  $\chi=\overline{\chi}_3^2$ (otherwise the sum is zero), where
$\chi^*$ denotes the mod $p^m$ extension of the Legendre symbol
(taking $\chi_2=\chi$, $\chi_1=\overline{\chi}$, $k=2$ we have $D=2$
and $\chi_4=\chi_0$ or $\chi^*$). For $m=1$ this is Han Di's
\cite[Lemma 1]{HanDi}. Cases where we can write the exponential sum
explicitly in terms of Gauss sums seem rare. Best known (after the
quadratic Gauss sums) are perhaps the Sali\'{e} sums, evaluated by
Sali\'{e} \cite{Salie} for $m=1$ (see Williams
\cite{salie1},\cite{salie2} or Mordell \cite{mordell} for a short
proof) and Cochrane \& Zheng \cite[\S 5]{esrf} for $m\geq 2$; for
$p\nmid AB$
$$\sumstar{x=1}{p^m}
\chi^*(x) e_{p^m}(Ax+Bx^{-1})=\chi^*(B)\begin{cases}
p^{\frac{1}{2}(m-1)}(e_{p^m}(2\gamma)
+e_{p^m}(-2\gamma))G\left(\chi^*,p\right), & \text{$m$ odd,} \\
p^{\frac{1}{2}m}\left(\chi^*(\gamma)e_{p^m}(2\gamma)
+\chi^*(-\gamma) e_{p^m}(-2\gamma)\right), & \text{$m$ even,}
\end{cases} $$ if $AB=\gamma^2$ mod $p^m$, and zero if
$\chi^*(AB)=-1$. Cochrane \& Zheng's $m\geq 2$ method works with a
general $\chi$  as long as their critical point quadratic congruence
does not have a repeat root, but formulae seem lacking when  $m=1$
and $\chi\neq \chi^*$.

For the Jacobsthal sums we get (essentially Theorems 6.1.14 \&
6.1.15 of \cite{BerndtBk})
\begin{align*}  \sum_{m=1}^{p-1}
\left(\frac{m}{p}\right)\left( \frac{m^k+B}{p}\right)
 & =\left(\frac{B}{p}\right)\sum_{j=0}^{k-1}
\chi(B)^{2j+1}\frac{G(\chi^{2j+1},p)G(\overline{\chi}^{2j+1}\chi^*,p)}{G(\chi^*,p)},\\
\sum_{m=0}^{p-1} \left( \frac{m^k+B}{p}\right) & = \left(\frac{B}{p}\right)\sum_{j=1}^{k-1} \chi(B)^{2j}\frac{G(\chi^{2j},p)G(\overline{\chi}^{2j}\chi^*,p)}{G(\chi^*,p)},
\end{align*}
when $p\equiv 1$ mod $2k$ and $p\nmid B$, where $\chi $ denotes a
mod $p$ character of order
 $2k$ and $\chi^*$ the mod $p$ character corresponding to the Legendre symbol (see also \cite{WilliamsJac}).

\begin{proof} Observe that if $\chi$ is a primitive character mod
$p^j$, $j\geq 1$,  then \begin{equation} \label{gauss1}
\sum_{y=1}^{p^j} \chi(y) e_{p^j}(Ay) = \overline{\chi}(A)
G(\chi,p^j). \end{equation}  Indeed, for $p\nmid A$  this is plain
from $y\mapsto A^{-1}y$. If $p\mid A$ and $j=1$ the sum equals
$\sum_{y=1}^{p}\chi(y)=0$ and for $j\geq 2$ writing
$y=a^{u+\phi(p^{j-1})v}$, $a$ a primitive root mod $p^m$, $\chi
(a)=e_{\phi(p^j)}(c)$,  $u=1,...,\phi(p^{j-1}),$ $v=1,..,p$, \be
\label{gaussprop} \sum_{y=1}^{p^j} \chi(y)
e_{p^j}(Ay)=\sum_{u=1}^{\phi (p^{j-1})} \chi(a^u)e_{p^j}(Aa^u)
\sum_{v=1}^p e_{p}(cv)=0. \ee Hence if $\chi_2$ is a primitive
character mod $p^m$ we have
$$G(\overline{\chi_2},p^m) \chi_2
(Ax^k+B)=\sum_{y=1}^{p^m}\overline{\chi_2}(y)e_{p^m}((Ax^{k}+B)y)$$
and, since $\chi_1=\chi_3^k$ and $D=(k,\phi (p^m))$,
\begin{align*} G(\overline{\chi_2},p^m)\sum_{x=1}^{p^m} \chi_1(x)
\chi_2(Ax^k+B) & =\sum_{x=1}^{p^m} \chi_3(x^k) \sum_{y=1}^{p^m}\overline{\chi_2}(y)e_{p^m}((Ax^{k}+B)y) \\
 & =\sum_{x=1}^{p^m} \chi_3(x^D) \sum_{y=1}^{p^m}\overline{\chi_2}(y)e_{p^m}((Ax^{D}+B)y) \\
  & = \sum_{\chi_4^D=\chi_0} \sum_{u=1}^{p^m} \chi_3(u) \chi_4 (u)
  \sum_{y=1}^{p^m}\overline{\chi_2}(y)e_{p^m}((Au+B)y) \\
   & =  \sum_{\chi_4^D=\chi_0} \sum_{y=1}^{p^m}\overline{\chi_2}(y)e_{p^m}(By)\sum_{u=1}^{p^m}
\chi_3\chi_4 (u)e_{p^m}(Auy) \\
  & =  \sum_{\chi_4^D=\chi_0} \sum_{y=1}^{p^m}\overline{\chi_2\chi_3\chi_4}(y)e_{p^m}(By)
 \sum_{u=1}^{p^m}
\chi_3\chi_4 (u)e_{p^m}(Au).
\end{align*}
Since $p\nmid B$ we have
$$  \sum_{y=1}^{p^m}\overline{\chi_2\chi_3\chi_4}(y)e_{p^m}(By) =   \chi_2\chi_3\chi_4 (B) G(\overline{\chi_2\chi_3\chi_4},p^m). $$
If $\chi_3\chi_4$ is a mod $p^{m-n}$ character then
$$\sum_{u=1}^{p^m}
\chi_3\chi_4 (u)e_{p^m}(Au)= p^n\sum_{u=1}^{p^{m-n}}\chi_3\chi_4
(u)e_{p^{m-n}}(A'u) =p^n \overline{\chi_3\chi_4}(A')
G(\chi_3\chi_4,p^{m-n}). $$ If $\chi_3\chi_4 $ is a primitive
character mod $p^j$ for some $m-n <j \leq m$ then by
\eqref{gaussprop}
$$\sum_{u=1}^{p^m}
\chi_3\chi_4 (u)e_{p^m}(Au)= p^{m-j}\sum_{u=1}^{p^{j}}\chi_3\chi_4
(u)e_{p^{j}}(p^{j-(m-n)}A'u) =0, $$ and the result follows.
\end{proof}

Notice that if $m\geq n+2$ then by \eqref{gaussabs} the set $X$ can
be further restricted to those $\chi_4$ with $\chi_3\chi_4 $
primitive mod $p^{m-n}$. Hence if $ p^t || \,k$, with $m\geq n+t+2 $
and we write $\chi_3(a)=e_{\phi (p^m)}(c_3)$, $\chi_4(a)=e_{\phi
(p^m)}(c_4)$ we have $p^{m-1-t} \mid  \,c_4$, $p^n ||\, (c_3+c_4)$,
giving $p^n || \, c_3$. From $\chi_3^{k}=\chi_1=\chi^l$ this yields
$p^{n+t}||\,c_3 k=c_1=cl$ and $p^{n+t}||\,l$. If $n>0$ we deduce
that $p^t||\; l+wk$. Moreover when $n=0$ reversing the roles of $A$
and $B$ gives $p^t ||\; l+wk$. Hence when $m\geq n+t+2$ we have
$S(\chi_1,\chi_2,Ax^k+B,p^m)=0$ unless \eqref{conditions1} holds.
For $m=n+t+1$ we similarly still have $p^{n+t}\mid l$.

\section{Evaluation of the Sums}
\begin{theorem} \label{main} Suppose that $p$ is an odd prime and  $\chi_1$, $\chi_2$ are mod $p^m$ characters with $\chi_2$ primitive.

If $\chi_1$ satisfies \eqref{conditions}, and \eqref{chareq0} has a
solution $x_0$ satisfying \eqref{conditions2}, then
$$
 \sum_{x=1}^{p^m}\chi_1(x)\chi_2(Ax^k+B) = d\chi (g(x_0))\begin{cases} p^{m-1}, & \text{if $t+n+1 < m \leq 2t+n+2$,} \\
 p^{\frac{m+n}{2}+t}, &  \text{if $m > 2t+n+2$, $m-n$ even,}\\
 p^{\frac{m+n}{2}+t} \varepsilon_1, &  \text{if $m > 2t+n+2$, $m-n$ odd,} \end{cases}
$$
where $n$, $d$, $t$ and  $g$ are as defined in \eqref{defn},
\eqref{defdt} and \eqref{defg}, with
$$\varepsilon_1=\left( \frac{\alpha}{p} \right) e_p\left(-2^{-2}\beta^2 \alpha^{-1}\right) \, \varepsilon, \hspace{3ex} \varepsilon = \begin{cases} 1 & p \equiv 1 \text{ mod } 4, \\
i & p \equiv 3 \text{ mod } 4, \end{cases}
$$
where  $\alpha$ and $\beta$ are integers defined in
\eqref{defalphabeta}  below and $\left(\frac{\alpha}{p}\right)$ is
the Legendre symbol.

If $\chi_1$ does not satisfy \eqref{conditions}, or
\eqref{chareq0} has no solution satisfying \eqref{conditions2}, then
the sum is zero.
\end{theorem}

For the mod $p^m$ Jacobi sums, $\chi_1=\chi^l$, $\chi_2=\chi^w$,
$\chi$ primitive mod $p^m$ with $p\nmid lw(l+w)$, we have
$x_0=l(l+w)^{-1}$ and
$$ \sum_{x=1}^{p^m}\chi_1(x)\chi_2(1-x) =
\frac{\chi_1(l)\chi_2(w)}{\chi_1\chi_2 (l+w)} p^{\frac{m}{2}} \begin{cases} 1, & \text{ if $m$ is even,}\\
 \left(\frac{-2rc}{p}\right) \left(\frac{lw(l+w)}{p}\right)\varepsilon, & \text{ if $m\geq 3$ is odd,} \end{cases}$$
 with $r$ and $c$ as in \eqref{defr} and \eqref{defc} below.

\begin{proof} Let $a$ be a primitive root mod $p$ such that
\be \label{defr} a^{\phi(p)}=1+rp,\;\;\; p\nmid r. \ee
Thus $a$ is a
primitive root for all powers of $p$. We define the integers $r_l$,
$p\nmid r_l$, by
$$
a^{\phi(p^l)}=1+r_l p^l,
$$
so that $r=r_1$. Since $(1+r_{s+1}p^{s+1})=(1+r_sp^s)^p$, for any
$s\geq 1$ we have
 \be \label{rcong}
r_{s+1}\equiv r_s \text{ mod } p^s. \ee
% and for
%$p>3$
% \be \label{rcong2} r_{s+1}\equiv r_s  +\frac{1}{2}(p-1)r_s^2 p^{s} \text{
%mod } p^{2s}. \ee

We define the integers $c$, $c_1=cl$, $c_2=cw$,  by \be \label{defc}
\chi(a)=e_{\phi(p^m)}(c),\;\;\;
\chi_1(a)=e_{\phi(p^m)}(c_1),\;\;\;\chi_2(a)=e_{\phi(p^m)}(c_2). \ee
Since $\chi_2$ is assumed primitive we have $p\nmid c_2$.

We write
%$$ \gamma = u\frac{\phi (p^L)}{d} + v $$
%where
$$ \gamma = u\frac{\phi (p^L)}{d} + v,\;\;\;\;\; L:= \begin{cases}  1, & \text{ if } m\leq n+2t+2, \\
  \left\lceil \frac{m-n}{2}\right\rceil -t,  & \text{ if $m>n+2t+2$, }\end{cases} $$
and observe that if $u=1,...,dp^{m-L}$ and $v$ runs through an
interval $I$ of length $\phi(p^L)/d$ then $\gamma$ runs through a
complete set of residues mod $\phi (p^m)$. Hence setting
$h(x)=Ax^{k}+B$ and writing $x=a^{\gamma}$ we have
\begin{align*}
\sum_{x=1}^{p^m}\chi_1(x)\chi_2(h(x)) &
% =\sum_{\gamma=1}^{\phi(p^m)} \chi_1(a^{\gamma})\chi_2(h(a^\gamma))\\&
 =\sum_{v\in I} \chi_1(a^{v}) \sum_{u=1}^{dp^{m-1}}\chi_1 (a^{u\frac{\phi (p^L)}{d}})\chi_2\left(h\left(a^{u\frac{\phi (p^L)}{d}+v}\right)\right). \\
\end{align*}
Since $2(L+t)+n\geq m$ we can write
\begin{align*}
h\left(a^{u\frac{\phi (p^L)}{d}+v}\right) & = A \left( a^{\phi (p^{L+t})}\right)^{u\left(\frac{k}{dp^t}\right)}a^{vk}  + B
= A \left( 1+ r_{L+t}p^{L+t}\right)^{u\left(\frac{k}{dp^t}\right)}a^{vk} +B \\
 & \equiv h(a^v)+ A'u\left(\frac{k}{dp^t}\right)a^{vk} r_{L+t}p^{L+t+n} \text{ mod } p^m. \\
\end{align*}
This is zero mod $p$ if $p\mid h(a^v)$ and consequently any such $v$
give no contribution to the sum. If $p\nmid h(a^v)$ then, since
$r_{L+t}\equiv r_{L+t+n}$ mod $p^{L+t}$,
\begin{align*}
h\left(a^{u\frac{\phi (p^L)}{d}+v}\right)  & \equiv h(a^v) \left( 1
+A'u\left(\frac{k}{dp^t}\right)h(a^v)^{-1}a^{vk} r_{L+t+n}p^{L+t+n}
\right) \text{ mod } p^m \\
 & \equiv h(a^v)
a^{A'u\left(\frac{k}{dp^t}\right)h(a^v)^{-1}a^{vk}\phi(p^{L+t+n})}
\text{ mod } p^m.
\end{align*}
Thus,
\begin{align*}
\sum_{x=1}^{p^m} \chi_1(x)\chi_2(h(x)) & = \sum_{\substack{v\in I \\ p\nmid h(a^v)}} \chi_1(a^{v})
\chi_2(h(a^v)) \sum_{u=1}^{dp^{m-L}}\chi_1\left(a^{u\frac{\phi (p^L)}{d}}\right) \chi_2\left(
a^{u \frac{\phi (p^L)}{d} Ak a^{vk} h(a^v)^{-1}}\right), \\
\end{align*}
where the inner sum $\displaystyle
\sum_{u=1}^{dp^{m-L}}e_{dp^{m-L}}\left(u \left(c_1 +c_2 A
h(a^v)^{-1}ka^{vk}\right)\right)$ is $dp^{m-L}$ if \be \label{cong1}
c_1+c_2h(a^v)^{-1}A'a^{vk}\left(\frac{k}{dp^t}\right)dp^{t+n} \equiv
0 \text{ mod } dp^{m-L} \ee and zero otherwise. Thus our sum will be
zero unless \eqref{cong1} has a solution with $p\nmid h(a^v)$. For
$m\geq n+t+1$ we have $m-L\geq t+n$ and a solution to \eqref{cong1}
necessitates $dp^{t+n}\mid c_1$ (giving us condition
\eqref{conditions}) with $p^{t+n} \mid\mid l$ for $m> n+t+1$. Hence
for $m
>n+t+1$
we can simplify the congruence to
\begin{equation}
h(a^v)\left(\frac{c_1}{dp^{t+n}}\right)+c_2A'a^{vk}\left(\frac{k}{dp^t}\right) \equiv 0 \text{ mod } p^{m-L-t-n}
\end{equation}
and for a solution we must have $p^t \mid\mid c_1+kc_2$.
Equivalently, \be \label{chareqn1}\frac{cg'(a^v)}{dp^{t+n}} \equiv 0
\text{ mod } p^{m-t-n-L}, \ee and the characteristic equation
\eqref{chareq0} must have a solution satisfying \eqref{conditions2}.
Suppose that \eqref{chareq0} has a solution $x_0=a^{v_0}$ with
$p\nmid h(x_0)$ and that $m>n+t+1$. Rewriting the congruence
\eqref{chareqn1} in terms of the primitive root, $a$, gives
$$
a^{vk} \equiv a^b \text{ mod } p^{m-t-n-L}
$$
for some integer $b$.  Thus two solutions to \eqref{chareqn1},
$a^{v_1}$ and $a^{v_2}$ must satisfy
$$v_1k\equiv v_2 k \text{ mod }
\phi(p^{m-t-n-L}). $$ That is $v_1 \equiv v_2$ mod $\frac{(p-1)}{d}$
if $m\leq n+2t+2$ and if $m>n+2t+2$
$$ v_1\equiv v_2 \text{ mod }
\frac{ \phi (p^{m-n-2t-L})}{d} $$ where $m-n-2t-L=L$ if $m-n$ is
even and $L-1$ if $m-n$ is odd. Thus if $n+t+1 < m\leq n+2t+2$ or
$m>n+2t+2$ and $m-n$ is even our interval $I$ contains exactly one
solution $v$. Choosing $I$ to contain $v_0$ we get that
$$
\sum_{x=1}^{p^m} \chi_1(x)\chi_2(h(x)) = dp^{m-L}\chi_1(x_0)\chi_2(h(x_0)).
$$

Suppose that  $m>n+2t+2$ with $m-n$ odd and set
$s:=\frac{m-n-1}{2}$. In this case $I$ will have $p$ solutions and
we pick our interval $I$ to contain the $p$ solutions
$v_0+yp^{s-t-1}\left(\frac{p-1}{d}\right)$ where $y=0,...,p-1$.
Since $dp^t\mid c_1$ and $dp^t\mid k$ we can write, with $g$ defined
as in \eqref{defg},
$$ g_1(x):=g(x)^c=x^{c_1}(Ax^k+B)^{c_2}=:H\left(x^{dp^t}\right). $$
Thus, setting $\chi=\chi_4^c$, where $\chi_4$ is the mod $p^m$
character with $\chi_4(a)=e_{\phi (p^m)}(1)$,
\begin{align*} \sum_{x=1}^{p^m} \chi_1(x)\chi_2(h(x))
 & =dp^{\frac{m+n-1}{2}+t} \sum_{y=0}^{p-1}
\chi\left(g\left(a^{v_0+yp^{s-t-1}\left(\frac{p-1}{d}\right)}\right)\right)\\
 & = dp^{\frac{m+n-1}{2}+t}  \sum_{y=0}^{p-1}
\chi_4\left( H\left( x_0^{dp^t} a^{y\phi \left(p^{s}\right)}\right) \right),
\end{align*}
where \be \label{cong} x_0^{dp^t}a^{y\phi \left(p^{s}\right)}=
x_0^{dp^t}\left(1+r_s p^{s} \right)^y = x_0^{dp^t} + yr_s
x_0^{dp^t}p^{s}  \text{ mod } p^{m-n-1}. \ee
Since
$$ p^{-n} H'(x^{dp^t}) =
\left(\frac{xg_1'(x)}{dp^{t+n}}\right) x^{-dp^t} \in \mathbb Z [x],
$$ we have $p^n \mid{\frac{ H^{(k)}\left(x_0^{dp^t}\right)}{k!}}$, for all $k\geq
1.$  As  $xg_1'(x)=(c_1+kc_2)g_1(x)- c_2kBg_1(x)/h(x)$,
$$p^{-n} H''(x^{dp^t})x^{2dp^{t}}=\left(\frac{c_1}{dp^{t}}+c_2\frac{k}{dp^{t}} -c_2\frac{k}{dp^{t}} \frac{B}{h(x)} -1\right)
\left(\frac{xg_1'(x)}{dp^{t+n}}\right)
 +
c_2\left(\frac{k}{dp^t}\right)^2A'Bx^{k} \frac{g_1(x)}{h(x)^2}. $$
Plainly a solution $x_0$ to  \eqref{chareq0} satisfying
\eqref{conditions2} also has $g_1'(x_0)\equiv 0$ mod
$p^{\frac{m+n-1}{2}+t}$ and  \be \label{deflambda}
\frac{x_0g_1'(x_0)}{dp^{t+n}} =\lambda
p^{\frac{m-n-1}{2}},\hspace{3ex} H'(x_0^{dp^t}) =x_0^{-dp^t} \lambda
p^{\frac{m+n-1}{2}}, \ee  for some integer $\lambda$, and
$$ p^{-n}H''(x_0^{dp^t}) \equiv  c_2 \left( \frac{k}{dp^{t}}\right)^2 A'Bx_0^{k-2dp^t}\frac{g_1(x_0)}{h(x_0)^2} \text{ mod
} p.
$$ Hence by the Taylor expansion, using \eqref{cong} and that  $r_{s}\equiv r_{m-1}\equiv r$ mod $p$,
\begin{align*} H\left( x_0^{dp^t} a^{y\phi
\left(p^{s}\right)}\right) & \equiv H(x_0^{dp^t}) +
H'(x_0^{dp^t}) yr_s x_0^{dp^t}p^{\frac{m-n-1}{2}} +
2^{-1}H''(x_0^{dp^t}) y^2r_s^2
x_0^{2dp^t}p^{m-n-1}  \text{ mod } p^m \\
 & \equiv g_1(x_0) \left( 1 + \left(\beta  y + \alpha y^2\right)r_{m-1}p^{m-1}\right) \text{ mod } p^m \\
 & \equiv g_1(x_0) a^{\left(\beta  y + \alpha y^2\right)\phi (p^{m-1})} \text{ mod } p^m, \\
\end{align*}
with \be \label{defalphabeta} \beta := g_1(x_0)^{-1}\lambda
,\hspace{3ex} \alpha :=  2^{-1} c_2h(x_0)^{-2} r A'B
\left(\frac{k}{dp^t}\right)^2 x_0^k, \ee and
$$\chi_4\left( H \left( x_0^{dp^t} a^{y\phi
\left(p^{s}\right)}\right)
\right)= \chi (g(x_0)) e_p(\alpha y^2 +\beta y). $$ Since
plainly $p\nmid \alpha$, completing the square then gives the result
claimed
\begin{align*} \sum_{x=1}^{p^m} \chi_1(x)\chi_2(h(x)) & =dp^{\frac{m+n-1}{2}+t}\chi (g(x_0))e_p(-4^{-1} \alpha^{-1}\beta^2)\sum_{y=0}^p  e_p(\alpha
y^2)\\
 &   = dp^{\frac{m+n-1}{2}+t}\chi (g(x_0))
 e_p(-4^{-1} \alpha^{-1}\beta^2)\left(\frac{\alpha}{p}\right)\varepsilon p^{\frac{1}{2}} \end{align*}
where $\varepsilon$ is $1$ or $i$ as $p$ is $1$ or  $3$ mod $4$.
Notice that if $x_0$ is a solution to the stronger congruence
$g'(x_0)\equiv 0$ mod $p^{\left[ \frac{m+n}{2}\right]+t+1}$ then
$\beta=0$ and the $e_p(-4^{-1} \alpha^{-1}\beta^2)$ can be omitted.
\end{proof}

\end{document}